\documentclass[11pt]{amsart}

\usepackage{epsfig}
\usepackage{amsfonts}
\usepackage{amsmath}
\usepackage{array}
\usepackage{tikz}
\usepackage[bookmarks]{hyperref}

\newcommand{\stab}[1]{G_{#1}}
\newcommand{\third}[1]{ \overline{-({#1}) } }
\newcommand{\br}[1]{\left\{#1\right\}}
\newcommand{\pr}[1]{\left(#1\right)}
\newcommand{\ZZ}{\mathbb{Z}}

\newtheorem{theorem}{Theorem}[section]
\newtheorem{lemma}[theorem]{Lemma}
\newtheorem{observation}[theorem]{Observation}
\newtheorem{proposition}[theorem]{Proposition}
\newtheorem{claim}{Claim}
\theoremstyle{definition}
\newtheorem{definition}[theorem]{Definition}

\begin{document}

\title{A New Proof of Kemperman's Theorem}

\author[Boothby]{Tomas Boothby}
\author[DeVos]{Matt DeVos}
\thanks{The second author is supported in part by an NSERC Discovery Grant (Canada).}
\author[Montejano]{Amanda Montejano}
\thanks{The third author was partially supported by the project PAPIT IA102013.}
\date{}

\begin{abstract}
Let $G$ be an additive abelian group and let $A,B \subseteq G$ be finite and nonempty.  The pair $(A,B)$ is called critical if the sumset $A+B = \{ a+b \mid \mbox{$a \in A$ and $b\in B$} \}$ satisfies $|A+B| < |A| + |B|$.  Vosper proved a theorem which  characterizes all critical pairs in the special case when $|G|$ is prime.  Kemperman generalized this by proving a structure theorem for critical pairs in an arbitrary abelian group.  Here we give a new proof of Kemperman's Theorem.  
\end{abstract}

\maketitle

\section{Introduction}

Throughout this paper we shall assume that $G$ is an additive abelian group. For subsets $A,B \subseteq G$, we define the \emph{sumset} of $A$ and $B$ to be  $A + B = \{ a + b \mid \mbox{$a \in A$ and $b \in B$} \}.$  If $g \in G$ we let $g + A = \{g\} + A$ and $A + g = A + \{g\}$.  The \emph{complement} of $A$ is the set $\overline{A} = G \setminus A$, and we let $-A = \{ -a \mid a \in A \}$.  

The classical direct problem for addition in groups is to ask: how small the sumset $A + B$ can be? If $G \cong \ZZ$ (or more generally, $G$ is torsion-free) it is not difficult to argue that  $| A + B | \geq | A | + | B | - 1$ holds for every pair of finite nonempty sets $(A,B)$.  In 1813 Cauchy proved that this assertion remains true when the order of $G$ is prime and $ A + B \neq G$. This  result was rediscovered by Davenport in 1935, and it is now known as the Cauchy-Davenport theorem.

\begin{theorem}[Cauchy \cite{cauchy} - Davenport \cite{davenport}]
\label{thm:CD}
If $A,B \subseteq \ZZ / p\ZZ$ are nonempty and $p$ is prime then 
\[ |A+B| \ge \min \{ p, |A| + |B| - 1 \}. \]
\end{theorem}

For arbitrary abelian groups we can not expect to have such a lower bound. For instance, if $H$ is a finite proper nontrivial subgroup of $G$, and $A=B = H$, then we will have $A+B=H$. So any generalization of Theorem \ref{thm:CD} will have to take subgroup structure into account.  Next we introduce an important theorem of Kneser which yields a generalization of Cauchy-Davenport to arbitrary abelian groups.

We define the \emph{stabilizer} of a subset  $A\subseteq G$, denoted  $\stab{A}$, to be the subgroup of $G$ defined by $\stab{A}= \{ g \in G \mid g + A = A \}.$  Note that $A$ is a union of $\stab{A}$-cosets, and $\stab{A}$ is the maximal subgroup of $G$ with this property.  For a subgroup $H\leq G$, we say that a subset  $A$ is \emph{$H$-stable} if $A+H=A$ (equivalently, $H \le \stab{A}$).

\begin{theorem}[Kneser \cite{kneser}, version I]
\label{thm:Kneser}
If $A$ and $B$ are finite nonempty subsets of $G$ and $H = \stab{A+B}$, then
\begin{equation} \label{eq:Kneser}
|A+B| \ge |A + H| + |B+H| - |H|.
\end{equation} 
\end{theorem}

To better understand Kneser's theorem, let us introduce some further notation.   Whenever $H \le G$ we let $\varphi_{G/H}$ denote the canonical homomorphism from $G$ to the quotient group $G/H$.  Now for $H = \stab{A+B}$ let  $\tilde{A}=\varphi_{G/H}(A)$ and  $\tilde{B}=\varphi_{G/H}(B)$.  By definition we have $|A+B|=|\tilde{A}+\tilde{B}||H|$, $|A+H|=|\tilde{A}||H|$ and  $|B+H|=|\tilde{B}||H|$. Using these simple equalities, we can express (\ref{eq:Kneser}) as $|\tilde{A}+\tilde{B}|\geq |\tilde{A}|+|\tilde{B}|-1$, acquiring the appearance of Cauchy-Davenport's lower bound in $G / H$. 

Define the \emph{deficiency} of a pair $(A,B)$ to be $\delta(A,B) = |A| + |B| - |A+B|$. We will say that a pair $(A,B)$ is \emph{critical} if $\delta (A,B)> 0$. The Cauchy-Davenport Theorem implies that, apart from the case when $A+B= \ZZ / p\ZZ$, all critical pairs in $\ZZ / p\ZZ$ satisfy $\delta(A,B)=1$. Meanwhile, Kneser's theorem asserts that for a critical pair $(A,B)$ in $G$, the pair $(\tilde{A},\tilde{B})$ of $G/H$ as defined above, will be critical with deficiency $\delta(\tilde{A},\tilde{B})=1$.  Indeed, Kneser's Theorem is equivalent to the statement that every critical pair $(A,B)$ with $H = G_{A+B}$ satisfies $|A+B| = |A+H| + |B+H| - |H|$.

Now we shall turn our attention to the structure of critical pairs.  One simple construction for a critical pair $(A,B)$ is to choose $A,B$ so that $\min\{ |A|, |B| \} = 1$.  A second, more interesting construction is to choose $A$ and $B$ to be arithmetic progressions with a common difference.  In 1956 Vosper proved the following theorem which characterizes critical pairs in groups of prime order, and these structures feature prominently in his result.

\begin{theorem}[Vosper \cite{vosper1} \cite{vosper2}, version I]
\label{thm:vosper1}
If $(A,B)$ is a critical pair of nonempty subsets of $\ZZ/p\ZZ$ and $p$ is prime, then one of the following holds.
\begin{enumerate}
\item $|A| + |B| > p$ and $A+B = \ZZ/p\ZZ$.
\item $|A| + |B| = p$ and $|A+B| = p-1$.
\item $\min\{|A|, |B| \} = 1$.
\item $A$ and $B$ are arithmetic progressions with a common difference.
\end{enumerate}
\end{theorem}

In 1960 Kemperman proved a structure theorem which characterizes critical pairs in an arbitrary abelian group. Although this theorem was published few years after Vosper's, it took some time before it achieved the recognition and attention it deserved.  This resulted in part from the inherent complexity of critical pairs, and in part from the difficult nature of Kemperman's paper.  Recently, this situation has improved considerably thanks to the work of Grynkiewicz \cite{grynkiewicz-qpk} \cite{grynkiewicz-step}, Lev \cite{lev-kemp}, and Hamidoune \cite{hamidoune-kemperman} \cite{hamidoune-structure}.  Grynkiewicz recasts Kemperman's Theorem and then takes a step further by characterizing those pairs $(A,B)$ with $|A+B| = |A+ |B|$.  Lev gives a more convenient ``top-down'' formulation of Kemperman's Theorem which we shall adopt here.  Finally, Hamidoune showed that all of these results could be achieved using the isoperimetric method.  

Here we shall give a new proof of Kemperman's theorem based on some recent work of the second author which generalizes Kemperman's Theorem to arbitrary groups.  Although this generalization leans heavily on the isoperimetric method, we shall not adopt these techniques here.  Instead we will exploit Kneser's theorem, thus making our proof rather closer in spirit to Kemperman's original than to any of these more recent works.  Our paper also differs with the existing literature in our statement of Kemperman's Theorem.  The main difference here is that we will work with triples of subsets instead of pairs, and this has the effect of reducing the number of configurations we need to consider.
 
The remainder of this paper is organized as follows.  Over the next two sections, we reduce the original classification problem to a classification problem for certain types of triples of subsets.  Section~\ref{sec:critrios} contains our new statement of Kemperman's theorem, and the remaining sections are devoted to its proof.

\section{Pure Pairs}

We define a pair $(A,B)$ to be \emph{pure} if $G_A = G_B = G_{A+B}$.  Our main goal in this section is to reduce our original problem to that of classifying pure critical pairs. However, we shall first address some of the uninteresting constructions of critical pairs.  

Consider the behaviour appearing in the first outcome of Theorem~\ref{thm:vosper1}, in the context of a general abelian group.  If $A,B \subseteq G$ satisfy $|A| + |B| > |G|$, then every $g \in G$ satisfies $B \cap (g-A) \neq \emptyset$ , and it follows that $A+B  = G$.  So every such pair will be critical.  Therefore, the critical pairs $(A,B)$ with $A+B = G$ are precisely those for which $|A| + |B| > |G|$.  Accordingly, we will call such pairs \emph{trivial}.  Another rather uninteresting construction of a critical pair $(A,B)$ is to take exactly one of $A$ or $B$ to be empty.  So, we will also call a pair $(A,B)$ \emph{trivial} if either $A = \emptyset$ or $B = \emptyset$, and we will generally restrict our attention to nontrivial critical pairs.

Now we turn our attention to the notion of pure.

\begin{observation}
\label{pureobs}
If $G_{A+B} = H$, then $(A+H, B+H)$ is pure.
\end{observation}

\begin{proof}
 This follows from $H \le G_{A+H} \le G_{A+H + B} = G_{A+B} = H$ and a similar chain of inequalities for $G_{B+H}$.
\end{proof} 

Note that by our discussion from the previous section, every pure critical pair $(A,B)$ satisfies $|A+B| = |A| + |B| - |G_{A+B}|$.  Next we will show that the problem of classifying critical pairs reduces to that of classifying pure critical pairs.  In short, critical pairs $(A,B)$ are at most $G_{A+B}$ elements away from a pure critical pair $(A^*,B^*)$ where $A \subseteq A^*$ and $B \subseteq B^*$.  This \emph{superpair} / \emph{subpair} relation is denoted $(A,B) \subseteq (A^*,B^*)$.

\begin{proposition}
For every nontrivial pair of finite subsets  $(A,B)$ of $G$ the following are equivalent.
\begin{enumerate}
\item The pair $(A,B)$ is critical.
\item There exists a pure critical superpair $(A^*,B^*) \supseteq (A,B)$ for which
$|A^* \setminus A| + |B^* \setminus B| < |G_{A^*+B^*}|$.  
\end{enumerate}
\end{proposition}

\begin{proof}
If (1) holds, then set $H = G_{A+B}$ and note that Observation \ref{pureobs} implies that $A^* = A+H$ and $B^* = B+H$ have $(A^*,B^*)$ pure and critical.  Thus $|A| + |B| > |A+B| = |A^* + B^*| = |A^*| + |B^*| - |H|$ so (2) holds.  

If (2) holds, then set $H = G_{A^* + B^*}$, let $z \in A^* + B^*$ and choose $a \in A^*$ and $b \in B^*$ with $a+b = z$.  Now, for every $h \in H$ the elements $a' = a+h$ and $b' = b-h$ satisfy $a' \in A^*$ and $b' \in B^*$ and $a' + b' = z$.  So, $z$ has at least $|H|$ distinct representations as a sum of an element in $A^*$ and an element in $B^*$.  It follows from this and $|A^* \setminus A| + |B^* \setminus B| < |H|$ that $A+B = A^* + B^*$.  This gives us $|A+B| = |A^* + B^*| = |A^*| + |B^*| - |H| > |A| + |B|,$ so $(A,B)$ is critical and (1) holds.
\end{proof}

In light of the above proposition, to classify all critical pairs, it suffices to classify the nontrivial pure critical pairs.

\section{Trios}
\label{sec:trio}

In the study of critical pairs, there is a third set which appears naturally in conjunction with $A$ and $B$, namely $C = \overline{-(A+B)}$.  For simplicity, let us assume for a moment that $G$ is finite and $(A,B)$ is critical.  Then we have
\begin{itemize}
\item $0 \not\in A+ B + C$
\item $|A| + |B| + |C| > |G|$.
\end{itemize}

In this case we see that the pair $(B,C)$ is critical since $B+C$ is disjoint from $-A$ (so $|B+C| \le |G| - |A| < |B| + |C|$) and similarly $(A,C)$ is critical.  So, in other words, taking the set $C$ as defined above gives us a triple of sets so that each of the three pairs is critical.  Accordingly we now extend our definitions from pairs to triples.  To allow for infinite groups we shall permit  sets which are infinite but cofinite.

\begin{definition}
 If $A,B,C \subseteq G$ satisfy $0 \not\in A + B + C$ and each of $A$, $B$, $C$ is 
either finite or cofinite, then we say that $(A,B,C)$ is a \emph{trio}.  
\end{definition}

\begin{definition}
 If $(A,B,C)$ is a trio and $n$ is the size of the complement of the largest of the sets $A,B,C$ and $\ell, m$ are the sizes of the other two sets, then we define the \emph{deficiency} of $(A,B,C)$ to be $\delta(A,B,C) = \ell + m - n$.  We say that $(A,B,C)$ is \emph{critical} if $\delta(A,B,C) > 0$.  In the case that $G$ is finite, we have $\delta(A,B,C) = |A|+|B|+|C| - |G|$.
\end{definition}

We say that a trio $(A,B,C)$ is \emph{trivial} if one of $A$, $B$, or $C$ is empty.   These definitions for trios naturally extend our notions for pairs.  More precisely, if $A,B \subseteq G$ are finite and  $C = \overline{-(A+B)}$ then $(A,B,C)$ is a trio, and we have
\begin{itemize}
\item $(A,B)$ is trivial if and only if $(A,B,C)$ is trivial.
\item $\delta(A,B) = \delta(A,B,C)$
\item $(A,B)$ is critical if and only if $(A,B,C)$ is critical.
\end{itemize}

Vosper's Theorem has a convenient restatement in terms of trios, as the extra symmetry in a trio eliminates one of the outcomes (and assuming nontriviality eliminates another).  

\begin{theorem}[Vosper, version II]
\label{vosper2}
If $(A,B,C)$ is a nontrivial critical trio in $\ZZ/p\ZZ$ and $p$ is prime, then one of the following holds.
\begin{enumerate}
\item $\min\{ |A|, |B|, |C| \} = 1$.
\item $A$, $B$, and $C$ are arithmetic progressions with a common difference.
\end{enumerate}
\end{theorem}

Similar to pairs, we define the \emph{supertrio} relation $(A,B,C) \subseteq (A^*,B^*,C^*)$ if $A \subseteq A^*, B \subseteq B^*$, and $C \subseteq C^*$, and call a trio $(A,B,C)$ \emph{maximal} if the only supertrio $(A^*, B^*, C^*) \supseteq (A,B,C)$ is $(A,B,C)$ itself.  Note that $(A,B,C)$ is maximal if and only if $C = \overline{-(A+B)}$ and $B = \overline{-(A+C)}$ and $A = \overline{-(B+C)}$.  The following proposition shows that pure critical pairs come from maximal critical trios.

\begin{proposition}
Let $A,B,C \subseteq G$ be nonempty and assume $A,B$ are finite.  Then the following are equivalent.
\begin{enumerate}
\item $(A,B)$ is a pure critical pair and $C = \overline{-(A+B)}$.
\item $(A,B,C)$ is a maximal critical trio.
\end{enumerate}
\end{proposition}

\begin{proof}
Assume that $(A,B)$ is pure and critical with $H = G_A = G_B = G_{A+B}$, and let $C = \overline{-(A+B)}$.  Suppose (for a contradiction) that $(A^*,B,C)$ is a supertrio with $A \subset A^*$.  Then $A^* + B \subseteq \overline{-C} = A+B$ so $A^* + B = A+B$, but then $|A^* + B| = |A + B| = |A| + |B| - |H| < |A^*| + |B| - |H|$ contradicts Kneser's Thoerem.  By a similar argument there is no trio $(A,B^*,C)$ with $B^* \supset B$, and thus $(A,B,C)$ is maximal.   

Next assume $(A,B,C)$ is maximal and critical, and note that whenever $H \le G_B$ we must also have $H \le G_A$ (otherwise $(A+H,B,C)$ contradicts maximality).  It follows that $G_A = G_B = G_C$.  Now $G_{A+B} = G_C$ implies that $(A,B)$ is pure, and $\delta(A,B) = \delta(A,B,C) > 0$ implies that $(A,B)$ is critical.  That $C = \third{A+B}$ follows from maximality.
\end{proof}

The above proposition further reduces the general classification problem to that of determining all maximal critical trios.  Next we give a version of Kneser's Theorem for trios which illustrates a key property of maximal critical trios, and prove the equivalence of the two versions.

\begin{theorem}[Kneser, version II]
\label{kneser2}
 If $(A,B,C)$ is a maximal critical trio in $G$, then $G_A = G_B = G_C$ and $\delta(A,B,C) = |G_A|$.
\end{theorem}

\begin{proof}[Proof of Equivalence]
 To see that version I implies version II, let $(A,B,C)$ be a maximal critical trio and note that the previous proposition implies that $(A,B)$ is pure and critical.  Thus $G_A = G_B = G_{A+B} = G_C$, and by Kneser's Theorem $\delta(A,B,C) = \delta(A,B) = |G_{A+B}|$.

 For the other direction, let $A,B \subseteq G$ be finite and nonempty and assume $(A,B)$ is critical (otherwise the result is trivial).  Set $C = \overline{-(A+B)}$ and choose a maximal supertrio $(A^*,B^*,C) \supseteq (A,B,C)$.  Now applying the theorem gives us a subgroup $H = G_C = G_{A^*} = G_{B^*}$ with $\delta(A^*, B^*, C) = |H|$.  Since $\delta(A,B,C) > 0$ we have $|A^* \setminus A| < |H|$ which implies $A^* = A+H$ and by a similar argument $B^* = B+H$.  Thus $|A+H| + |B+H| - |A+B| = \delta(A^*,B^*) = |H|$ as desired.
\end{proof}

\section{Critical Trios}\label{sec:critrios}

\newcommand{\impureH}{
 \fill[gray] (.1,.1) circle (4pt);
 \fill[gray] (1.1,.1) circle (2pt);
 \fill[gray] (2.1,.1) circle (5pt);
}
\newcommand{\pureH}{
 \foreach \x in {0,1,2} {
   \fill[gray] (\x,0) -- (\x + .2,0) -- (\x + .2,.2) -- (\x, .2);
 }
}
\newcommand{\triomaker}[2]{
 \begin{tikzpicture}[very thick,x=2cm,y=2cm]
  \foreach \x/\Y in {#1} {
    \foreach \y/\l in \Y {
      \fill[gray] (\x,\y * .2) -- (\x + .2, \y * .2) -- (\x + .2, \y * .2 + .2) -- (\x, \y * .2 + .2);
      \node [label=right:{$\l$}] (l) at (\x + .1, \y * .2 + .1){};
    }
  }
  #2
  \foreach \x in {0,1,2} {
    \foreach \y in {.2,.4,...,1.2} {
      \draw (\x,\y) -- (\x + .2,\y);
    }
    \node [label=right:{$H$}] (H) at (\x + .1, .1){};
    \draw (\x,0) -- (\x,1.4) -- (\x + .2, 1.4) -- (\x + .2, 0) -- cycle;
  }
  \node [label=below:{$A$}] (A) at (.1,0){};
  \node [label=below:{$B$}] (B) at (1.1,0){};
  \node [label=below:{$A+B$}] (C) at (2.1,0){};
 \end{tikzpicture}
}

\begin{figure}
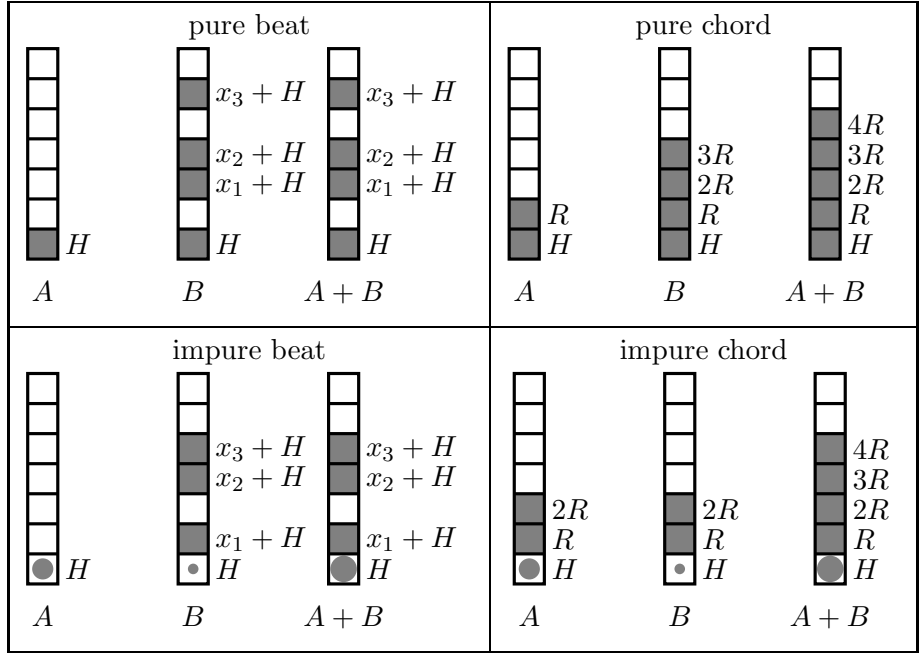

 \begin{tabular}{|c|c|}
  \hline
  pure beat & pure chord \\
  \triomaker{0/{}, 1/{2/x_1+H, 3/x_2+H,5/x_3+H}, 2/{2/x_1+H, 3/x_2+H,5/x_3+H}}{\pureH} & 
          \triomaker{0/{1/R}, 1/{1/R,2/2R,3/3R}, 2/{1/R,2/2R,3/3R,4/4R}}{\pureH} \\
  \hline
  impure beat & impure chord \\
  \triomaker{0/{}, 1/{1/x_1+H, 3/x_2+H,4/x_3+H}, 2/{1/x_1+H, 3/x_2+H,4/x_3+H}}{\impureH} &
          \triomaker{0/{1/R,2/2R}, 1/{1/R,2/2R}, 2/{1/R,2/2R,3/3R,4/4R}}{\impureH}\\
  \hline
 \end{tabular}

 \caption{Structure Atlas}\label{fig:atlas}
\end{figure}

Note that if $(A,B,C)$ is a trio, then any permutation of these three sets yields a new trio.  In addition, for every $g \in G$ we have that $(A+g, B-g, C)$ is a trio.  It follows immediately that these operations preserve nontriviality, maximality, criticality, and deficiency, and we say that two trios are \emph{similar} if one can be turned into the other by a sequence of these operations.  

Next we will introduce some terminology to describe the types of  behaviour present in the structure of nontrivial maximal critical trios.  We begin with a structure which generalizes those critical pairs $(A,B)$ with $\min\{ |A|, |B| \} = 1$ by allowing for subgroups.

\begin{definition} Let $H< G$ be finite.  A trio $\Upsilon$ is a \emph{pure beat relative to} $H$ if $\Upsilon$ is similar to a trio $(A,B,C)$ which satisfies the following: 
\begin{enumerate}
\item $A =H$,
\item   $\stab{B} = H$, and
\item $C = \overline{-(A+B)} \neq \emptyset$.
\end{enumerate}
\end{definition}

Before introducing our next structure, we require a bit more terminology.  Let $H < G$ be a finite subgroup, let $R \in G/H$ and assume that $G/H$ is a cyclic group generated by $R$.  Then we define any set of the form $S = \{ A + iR \mid 0 \le i \le k \}$ with $A \in G/H$ to be an $R$-\emph{sequence}.  We call $A$ the \emph{head} of this sequence, $A+kR$ the \emph{tail} of the sequence, and we say that $S$ is \emph{basic} if it has head $H$.  We define $k+1$ to be the \emph{length} of the sequence, and we call it \emph{nontrivial} if it has length at least $2$.  It is easy to see that a pair of $R$-sequences will be critical, and this is the form in which we will encounter arithmetic progressions.

\begin{definition} Let $H < G$ be finite with $G/H$ is cyclic.  A trio $\Upsilon$ is a \emph{pure chord relative to} $H$, if there exists $R \in G/H$ which generates $G/H$ and a trio $(A,B,C)$ similar to $\Upsilon$ for which the following hold.
\begin{enumerate}
\item $A,B$ are nontrivial $R$-sequences. \label{itm:nontriv}
\item $C = \overline{-(A+B)}$ is not contained in a single $H$-coset.  \label{itm:third}
\end{enumerate}
\end{definition}

It follows immediately from our definitions that every pure beat or pure chord relative to $H$ is a maximal critical trio with deficiency $|H|$.

For each of these two basic structures, there is a variant which allows for recursive constructions of maximal critical trios.  Before introducing these variants, we require another bit of terminology.  For every set $A \subseteq G$ there is a unique minimal subgroup $H$ for which $A$ is contained in an $H$-coset.  We denote this $H$-coset by $[A]$ and call it the \emph{closure} of $A$.

\begin{definition} A trio $\Upsilon$ is an \emph{impure beat}  \emph{relative to} $H < G$, if 
there is a trio $(A,B,C)$ similar to $\Upsilon$ for which

\begin{enumerate}
\item $[A] = H$,
\item $B \setminus H$ is $H$-stable,
\item $C \setminus H = \overline{-(A+B)}\setminus H$, and
\item $B \cap H \neq \emptyset$ and $C \cap H \neq \emptyset$.
\end{enumerate}
In this case $(A, B \cap H, C \cap H)$ is a trio in $H$ which we call a \emph{continuation} of $\Upsilon$.
\end{definition}

\begin{definition}
 Let $H < G$ be finite and assume $G/H$ is cyclic.  A trio $\Upsilon$ is an \emph{impure chord relative to} $H$, if there exists $R \in G/H$ which generates $G/H$ and a trio $(A,B,C)$ similar to $\Upsilon$ satisfying
 \begin{enumerate}
  \item $H \cup A$ and $H \cup B$ are nontrivial basic $R$-sequences,
  \item $C \setminus H = \overline{- (A+B)} \setminus H \neq \emptyset$, and
  \item $A \cap H$, $B \cap H$, and $C \cap H$ are all nonempty.
 \end{enumerate}
 As above, $(A \cap H, B \cap H, C \cap H)$  is a trio in $H$ which we call a \emph{continuation} of $\Upsilon$.
\end{definition}

Note that if $(A,B,C)$ is an impure beat or impure chord relative to $H$ and $(A',B',C')$ is a continuation, then our definitions imply that $(A',B',C')$ is a nontrivial trio in $H$.  Furthermore, it follows from these constructions that $(A',B',C')$ will be maximal whenever $(A,B,C)$ is maximal, and $\delta(A',B',C') = \delta(A,B,C)$.  

With this, we can finally state Kemperman's structure theorem.

\begin{theorem}[Kemperman]
\label{thm:kemperman}
Let $\Upsilon_1$ be a maximal nontrivial critical trio in $G_1$.  Then there exists a sequence of trios $\Upsilon_1, \Upsilon_2, \cdots, \Upsilon_m$ in respective subgroups $G_1 > G_2 > \cdots > G_m$ satisfying
\begin{enumerate}
\item $\Upsilon_i$ is an impure beat or an impure chord with continuation $\Upsilon_{i+1}$ for $1 \le i \le m-1$, and
\item $\Upsilon_m$ is either a pure beat or a pure chord.
\end{enumerate}
\end{theorem}

\section{Incomplete Closure}

In this section we focus our attention on critical pairs and trios which contain a set $A$ for which $[A] \neq G$.  In particular, we shall prove a stability lemma which shows that every maximal critical trio containing such a set must be a pure or impure beat.  We begin with a lemma which was proved for general groups by Olson \cite{olson}, but which follows from Kneser's Theorem for abelian groups (as observed by Lev \cite{lev-kemp}).

\begin{lemma}
\label{cor:kneser}
Let $A,B$ be nonempty finite subsets of $G$ and assume that $A+B \neq G$ and $[A] = G$.  Then
$|A+B| \ge \tfrac{1}{2}|A| + |B|$.
\end{lemma}

\begin{proof}
By Theorem~\ref{thm:Kneser}, $H=\stab{A+B}$ satisfies $|A+B| \geq |A+H|+|B+H|-|H|$ and $H \neq G$ since $A+B \neq G$.  Since $A$ is not contained in any $H$-coset, $|A+H| - |H| \ge \frac{1}{2}|A+H| \ge \frac{1}{2}|A|$.  
Combining these two inequalities yields the desired bound.
\end{proof}

For a set $A \subseteq G$ and a subgroup $H \le G$, we say that $A$ is $H$-\emph{quasistable} if there exists $R \in G/H$ so that $A \setminus R$ is $H$-stable.  Members of a pure beat or chord relative to $H<G$ are $H$-stable.  The impure versions comprise $H$-quasistable sets, and their continuations are the partial $H$ cosets.

\begin{lemma}
\label{lem:discon_sum}
Let $(A,B)$ be a critical pair of finite sets and assume $[A] \in G/H$ for some $H < G$.  Then $A+B$ is $H$-quasistable.  Furthermore, if $H$ is finite, then $\delta(H,B) \ge \delta(A,B)$.
\end{lemma}

\begin{proof}
 By replacing $A$ by $g+A$ for a suitable $g \in G$, we may assume that $[A] = H$.  Let $R_1, \ldots, R_k \in G/H$ be the $H$-cosets which have nonempty intersection with $B$, and for every $1 \le i \le k$ let $B_i = B \cap R_i$.  Now we have two inequalities,
 \begin{enumerate}
  \item $|A + B_i| \ge |B_i|$ and 
  \item if $A+B_i \neq R_i$, then $|A+B_i| \ge \frac{1}{2}|A| + |B_i|$,
 \end{enumerate}
 the second of which follows from the previous lemma.  Since $A+B$ is the disjoint union $\bigcup_{i=1}^k (A + B_i)$, it follows that there is at most one $1 \le i \le k$ for which $A+B_i \neq R_i$, so $A+B$ is $H$-quasistable.

 For the last part, we may assume that $H$ is finite and that $A+B_i = R_i$ for all $2 \le i \le k$.  Since $|A+B_1| \ge |A|$ we find \[ \delta(A,B) = |A| + |B| - \sum_{i=1}^k |A + B_i| \le |B| - (k-1)|H| = \delta(H,B) \] which completes the proof.
\end{proof}

We are now ready to prove our stability lemma for trios which contain a set with closure not equal to $G$.

\begin{lemma}[Beat Stability]
\label{lem:beat}
If $(A,B,C)$ is a maximal critical trio and $[A] \in G/H$ for some $H <G$, then $(A,B,C)$ is either a pure or impure beat.
\end{lemma}

\begin{proof} By possibly moving from $(A,B,C)$ to a similar trio, we may assume that $[A] = H < G$ and that $B$ is finite. Now, $(A,B)$ is a critical pair, so by Lemma~\ref{lem:discon_sum}, $A+B$ is $H$-quasistable.  If $A+B$ is $H$-stable, then $H$ is finite and it follows from maximality that $A=H$ and $H = \stab{B} = \stab{C}$ so $(A,B,C)$ is a pure beat.  In the latter case, we may assume (by possibly passing to a similar trio) that $\emptyset \neq (A+B) \cap H \neq H$ and it then follows from maximality that $(A,B,C)$ is an impure beat.
\end{proof}

\section{Purification}

In this section we will develop a process we call purification which will allow us to make a subtle modification to a critical trio to obtain a new trio with deficiency no smaller than the original.  This will be a key tool in the remainder of the paper.  

We have already defined notions of deficiency for pairs of finite sets and for trios.  It is also convenient to have a notion of deficiency for a single finite set.  If $\emptyset \neq A \subset G$ is finite we define the \emph{deficiency} of $A$ to be 
 \[ \delta(A) = \max_{B \subset G : A+B \neq G} \delta(A,B). \]
Here we only consider finite nonempty sets $B$.  Note that this is indeed well defined since for every $B \subseteq G$ we have $\delta(A,B) = |A| + |B| - |A+B| \le |A|$ so the maximum in the formula will be obtained.  The following theorem of Mann shows that there is always a finite subgroup which achieves this maximum. 

\begin{theorem}[Mann]
\label{thm:mann}
If $A \subset G$ is finite and nonempty, there exists a finite subgroup $H < G$ with $\delta(A,H) = \delta(A)$ and $A+H \neq G$.
\end{theorem}

\begin{proof}
Choose $\emptyset \neq B \subseteq G$ so that $\delta(A,B) = \delta(A)$ and $A+B \neq G$.  Now set $H = \stab{A+B}$ and apply Kneser's Theorem to obtain
\begin{eqnarray*}
\delta(A,B) 
   &=&    |A| + |B| - |A+B|	\\
   &\le&  |A| + |B| - |A+H| - |B+H| + |H| \\
   &\le&  |A| - |A+H| + |H| \\
   &=&    \delta(A,H).
\end{eqnarray*}
Finally, $A+H \subseteq A+B+H < G$ since $H = \stab{A+B}$.
\end{proof}

Next we establish a lemma which is a key part of purification.

\begin{lemma}
\label{inc_bd}
Let $H < G$ and $A \subset G$ be finite and assume $(A,H)$ is critical.  If $B \subseteq H$, then $\delta(A,B) \le \delta(A,H)$.
\end{lemma}

\begin{proof}
 We may assume that $(A,B)$ is critical, as otherwise the result holds immediately.  Choose $K \le G$ so that $[B] \in G/K$ and note that Lemma \ref{lem:discon_sum} implies $\delta(A,B) \le \delta(A,K)$.  Since $K \le H$, to complete the proof, it suffices to show $\delta(A,K) \le \delta(A,H)$ under the assumption $K < H$.
 
 Define $S = (A+H) \setminus A$ and let $S' = \{ g \in S \mid g+K \subseteq S \}$ and $S'' = S \setminus S'$.  Since $(A,H)$ is critical $|S'| < |H|$, and then we must have $|S'| \le |H| - |K|$ (since $|S'|$, $|H|$, and $|K|$ are all multiples of $|K|$).  Thus 
 \[ \delta(A,H) =  |H| - |S| \ge |K| - |S''| = |K| - |(A+K) \setminus A| = \delta(A,K) \]
 which completes the proof.
\end{proof}

\begin{lemma}[Purification]
\label{purification}
 Let $(A,B,C)$ be a critical trio in $G$, let $H \le G$, and assume $A$ and $H$ are finite and $(A,H)$ is critical.  If $R \in G/H$ satisfies $\emptyset \neq R \cap B \neq R$ and $S = \overline{-(A+R)}$, then $ \delta(A,B \cup R, C \cap S) \ge \delta(A,B,C)$.
\end{lemma}

\begin{proof}
Since $(A,B,C)$ and $(A,R,S)$ are trios, it follows that both $(A, B \cup R, C \cap S)$ and $(A, B \cap R, C \cup S)$ are trios.  Furthermore
\[ \delta(A, B \cup R, C \cap S) + \delta(A, B \cap R, C \cup S) = \delta(A,B,C) + \delta(A,R,S). \]
The previous lemma implies $\delta(A,R,S) = \delta(A,R) \ge \delta(A,B \cap R) \ge \delta(A,B \cap R, C \cup S)$ and together with the above equation, this yields the desired result.
\end{proof}

Note that the above lemma also applies to pairs. More precisely, if $A,B \subseteq G$ and $H < G$ are finite and both $(A,B)$ and $(A,H)$ are critical, then for every $R
\in G/H$ with $B \cap R \neq \emptyset$, we have $\delta(A,B \cup R) \ge \delta(A,B)$.

\section{Near Sequences}

The goal of this section is to establish two important lemmas concerning a type of set called a near sequence.  The first is a stability lemma which will show that whenever $(A,B,C)$ is a maximal critical trio with some additional properties, and $A$ is a near sequence, then $(A,B,C)$ must be a pure or impure chord.  The second will show that whenever $(A^*,B^*,C^*)$ is a pure or impure chord, of which $(A,B,C)$ is a critical subtrio, then every finite set among $(A,B,C)$ must be a near sequence.

We begin by introducing a couple of important definitions.  For this purpose we shall assume that $H < G$ is a finite subgroup and $R \in G/H$ generates the group $G/H$.

\begin{definition}
 We say that $A \subseteq G$ is a \emph{near} $R$-\emph{sequence} if $A+H$ is an $R$-sequence and $|(A+H) \setminus A| < |H|$.
\end{definition}

\begin{definition}
 We say that $A \subseteq G$ is a \emph{fringed} $R$-\emph{sequence} if
 \begin{enumerate}
  \item $A+H$ is an $R$-sequence, and
  \item if $A+H$ has head $S$ and tail $T$, then either $A \setminus S$ or $A \setminus T$ is $H$-stable.
 \end{enumerate}
\end{definition}

If $A$ is an $R$-sequence, near $R$-sequence, or fringed $R$-sequence, we say that $A$ is \emph{proper} if $| \overline{A} | \ge 2|H|$, and we call it \emph{nontrivial} if $|A| > |H|$.  Next we prove a technical lemma where fringed sequences emerge.

\begin{lemma}\label{lem:fringed}
 Let $(A,B)$ be a nontrivial critical pair, and assume that $A$ is a nontrivial near $R$-sequence for $R \in G/H$, and that $B$ is not contained in any $H$-coset.  If there exists an $R$-sequence $B^*$ with $B \subseteq B^*$ and $A+B^* \neq G$, then $A+B$ is a fringed  $R$-sequence.
\end{lemma}

\begin{proof} Suppose (for a contradiction) that the lemma fails, and let $A,B$ be a counterexample for which $|B|$ is minimum.  By shifting $A$ (i.e. replacing $A$ by $A+g$ for some $g \in G$) and $B$ we may assume that $A+H = \bigcup_{i=0}^{\ell} iR$ and $B^* = \bigcup_{i=0}^{ m} iR$.  For convenience let us define $A_i = A \cap i R$ and $B_i = B \cap i R$ for every $i \in \ZZ$.  By replacing $B^*$ with a smaller $R$-sequence, we may assume that $B_0 \neq \emptyset$ and $B_{ m} \neq \emptyset$.  We first prove a series of three claims.

\begin{claim}
 $B_i \neq \emptyset$ for $0 \le i \le  m$.
\end{claim}

It follows from repeatedly applying our purification lemma to $H$-cosets $R \in G/H$ for which $\emptyset \neq R \cap B \neq R$ that $(A,B+H)$ is critical and thus $(A+H,B+H)$ is critical.  It follows from this that the sets $\tilde{A},\tilde{B} \subseteq \ZZ$ given by $\tilde{A} = \{0,1,\ldots,{\ell} \}$ and $\tilde{B} = \{ i \in \ZZ \mid iR \cap B \neq \emptyset \}$ satisfy $(\tilde{A},\tilde{B})$ critical.  It follows, e.g. from Lemma 1.3 of Nathanson\cite{nathanson}, that $\tilde{B}$ is the interval $\{ 0, 1, \ldots  m\}$ which implies the claim.

\begin{claim}\label{clm:noguts}
 $A+B$ does not contain $\bigcup_{i=1}^{ \ell + m -1} iR$.
\end{claim}

Suppose for a contradiction that this claim fails.  Let $K_0 = \stab{A_0 + B_0}$ and $K_1 = \stab{A_{{\ell}} + B_{ m}}$.  We have $K_0, K_1 < H$ since $A+B$ is not a fringed sequence.  Now applying Kneser's Theorem to the sumsets $A_0 + B_0$ and $A_{{\ell}} + B_{ m}$ we find
\begin{eqnarray*}
|A+B| 
    &=&   ({\ell} + m - 1)|H| + |A_0 +B_0| + |A_{{\ell}} + B_{ m}|		\\
    &\ge& ({\ell} + m-1)|H| + |A_0| + |A_{\ell}| + |B_0| + |B_{ m}| - |K_0| - |K_1| \\
    &\ge& \pr{({\ell}-1)|H| + |A_0| + |A_{\ell}|} + \pr{( m-1)|H| + |B_0| + |B_{m}|} \\
    &=&   |A| + |B|
\end{eqnarray*}
which gives us the desired contradiction.

\begin{claim}
 $ m = 1$
\end{claim}

As usual, we suppose for contradiction that $ m > 1$.  First consider the set $B' = B \setminus B_{ m}$.  It follows easily from the inequality $|A_{\ell} + B_{ m}| \ge |B_{ m}|$ that $\delta(A,B') \ge \delta(A,B) > 0$ so $(A,B')$ is critical.   By the minimality of our counterexample, it follows that $A+B'$ must contain $\bigcup_{i=1}^{\ell+m-2} iR$.  Similarly, $(A,B\setminus B_0)$ is a critical pair and hence contains $\bigcup_{i=2}^{\ell+m-1} iR$.  Putting these together, we find $\bigcup_{i=1}^{\ell+m-1} iR \subseteq A+B$, contradicting Claim~\ref{clm:noguts}.

With these claims in place, we are ready to complete the proof.  By the purification lemma, $(A, B \cup R)$ and $(A, B \cup H)$ are critical.  This gives us
\begin{eqnarray}
0 < \delta(A,B \cup R) &=& |A| + |B_0| - {\ell} |H| - |A_0 + B_0| \label{eqn:fringed1} \\
0 < \delta(A,B \cup H) &=& |A| + |B_1| - {\ell} |H| - |A_{\ell} + B_1| \label{eqn:fringed2}  
\end{eqnarray}
We also have
 \begin{equation}
  |A| - |A_0 + B_0| - |A_{\ell} + B_1| \le |A| - |A_0| - |A_{\ell}|  \le ({\ell}-1)|H| \label{eqn:fringed3}
 \end{equation}
Now summing equations (\ref{eqn:fringed1}) and (\ref{eqn:fringed2}) and substituting (\ref{eqn:fringed3}) yields
 \begin{equation}
  0 < |A| + |B_0| + |B_1| - ({\ell}+1)|H|. \label{eqn:fringed4}
 \end{equation}
It follows from Claim~\ref{clm:noguts} that we may choose a point $z \in \pr{\bigcup_{i=1}^{\ell} iR } \setminus (A+B)$.  Assuming $z \in iR$, we have $B_0 \cap (z - A_i) = \emptyset$ and $B_1 \cap (z - A_{i-1}) = \emptyset$.  These together with $|A \setminus (A_{i-1} \cup A_i)| \le ({\ell}-1)|H|$ then imply
 \begin{equation}
  |B_0| + |B_1| \le 2 |H| - |A_{i-1}| - |A_i| \le ({\ell}+1)|H| - |A| \label{eqn:fringed5}
 \end{equation}
Inequalities (\ref{eqn:fringed4}) and (\ref{eqn:fringed5}) are contradictory, and this completes the proof.
\end{proof}

\begin{lemma}
Let $(A,B,C)$ be a nontrivial critical trio with $A,B$ finite, and assume that every supertrio $(A,B^*, C^*) \supseteq (A,B,C)$ has $(A,B,C) = (A,B^*,C^*)$.  Let $H < G$, let $R \in G/H$ and assume $A$ is a nontrivial proper near $R$-sequence.  If neither $B$ nor $C$ is contained in an $H$-coset, then $B$ and $\overline{C}$  are fringed $R$-sequences.
\end{lemma}

\begin{proof} \setcounter{claim}{0}
 Suppose (for a contradiction) that there is a counterexample to the lemma using the set $A$, and then choose $B$ and $C$ so that $(A,B,C)$ is a counterexample for which
 \begin{enumerate}
  \item $\delta(A,B,C)$ is maximum.
  \item $| \{ S \in G/H \mid \emptyset \neq S \cap B \neq S \} | + | \{ S \in G/H \mid \emptyset \neq S \cap C \neq S \} |$ is minimum. (subject to 1).
 \end{enumerate}
 Observe that if $B$ is a fringed $R$-sequence, we can automatically conclude that $\overline{C} = -(A+B)$ is a fringed $R$-sequence by the maximality of $C$.

 \begin{claim}\label{clm:Gfinite}
  There does not exist an $R$-sequence $D$ with $A+D \neq G$ so that $B \subseteq D$ or $C \subseteq D$.  So, 
  in particular, $G$ must be finite.
 \end{claim}

 If such a set $D$ exists with $B \subseteq D$, then by applying the previous lemma we deduce that $A+B$ is a fringed $R$-sequence.  But then, the maximality of $B$ implies that $B$ is a fringed $R$-sequence.  Similarly, if such a set $D$ exists with $C \subseteq D$, then $G$ is finite and $A+C$ is a fringed $R$-sequence.  But then $B = \third{A+C}$ is also a fringed $R$-sequence.

 \begin{claim}\label{clm:impure}
  There does not exist $S \in G/H$ with $S \subseteq B$ or with $S \subseteq C$.
 \end{claim}
 
 If $S \in G/H$ satisfies $S \subseteq B$ then Claim~\ref{clm:Gfinite} is violated by $D = \third{S+A}$ since $C \subset D$.  A similar argument holds if $S \subseteq C$.

 Since $G$ is finite by Claim~\ref{clm:Gfinite}, we will show that both $B$ and $C$ (hence $\overline{C}$) are fringed $R$-sequences.  Without loss of generality, $|C|\geq |B|$.  In particular, $|C|>|H|$ since $|G \setminus A| \geq 2|H|$ and $(A,B,C)$ is critical.

 Using Claim~\ref{clm:impure}, choose an $H$-coset $S \in G/H$ so that $\emptyset \neq B \cap S \neq S$.  Now setting $B' = B \cup S$ and $C' = C \cap \third{A+S}$,  Lemma~\ref{inc_bd} implies $\delta(A,B',C') \ge \delta(A,B,C)$ and it follows that $|C \setminus C'| \le |B' \setminus B| < |H|$.  Since $|C| > |H|$ we have that $C' \neq \emptyset$, so $(A,B',C')$ is a nontrivial critical trio.  Choose $B'',C''$ maximal so that $(A,B'',C'')$ is a supertrio of $(A,B',C')$.

 If $B'' \neq B'$ or $C'' \neq C'$ then $\delta(A,B'',C'') > \delta(A,B,C)$ so by the first criteria in our choice of counterexmple, the lemma holds for $(A,B'',C'')$.  On the other hand, if $B'' = B'$ and $C'' = C'$ then the quantity in our second optimization criteria improves, so again we find that the lemma holds for $(A,B'',C'')$.  

 If $C''$ is not contained in a single $H$-coset, then $B''$ is a fringed sequence, but then Claim~\ref{clm:Gfinite} is violated by $D = B'' + H$ since $B \subset D$.  So, we may assume that $C'' \subseteq T$ for some $T \in G/H$.  Now let $U \in G/H$ satisfy $U \subseteq \third{A+T}$ and suppose for contradiction that $B \cap U = \emptyset$.  In this case $(A, B' \cup U, C')$ is a trio and Lemma \ref{purification} implies
 \[ \delta(A, B', C') + |H| = \delta(A, B' \cup U, C') \le \delta(A, C') \le \delta(A,H) \le |H| \]
 which is a contradiction.  It follows that every $H$-coset contained in $\third{A+T}$ must have nonempty intersection with $B$. 

 With this knowledge, we now return to our original trio $(A,B,C)$ and modify it to form a new trio by setting $C''' = C \cup T$ and $B''' = B \cap \third{A + T}$.  It follows from our purification lemma that $(A,B''',C''')$ is a trio with $\delta(A,B''',C''') \ge  \delta(A,B,C)$.  Furthermore, since $\third{A+T}$ contains at least two $H$-cosets, the set $B'''$ cannot be contained in a single $H$-coset.  Now, by repeating the argument from above, we may extend $(A, B''', C''')$ to a trio with the second and third sets maximal, and then the lemma will hold for this new trio.  This then implies that $C'''$ is contained in an $R$-sequence which violates Claim~\ref{clm:Gfinite}.  This completes the proof.
\end{proof}

\begin{lemma}[Sequence Stability]
\label{lem:chord}
Let $(A,B,C)$ be a maximal critical trio with $[A] = [B] =[C] = G$.  If $A$ is a proper near sequence, then $(A,B,C)$ is either a pure or an impure chord.
\end{lemma}

\begin{proof}
 Let $A$ be a proper near $R$-sequence for $R \in G/H$ and assume (without loss) that $B$ is finite.  By the previous lemma we deduce that $B$ is a fringed $R$-sequence.  However, then by maximality $A$ is also a fringed $R$-sequence.  Again using maximality, we conclude that  $(A,B,C)$ is either a pure or impure chord relative to $H$.
\end{proof}

\begin{lemma}
\label{lem:near}
 Let $(A,B,C)$ be a critical trio of which $(A^*, B^*, C^*)$ is a maximal critical supertrio.  If $(A^*, B^*, C^*)$ is a pure or impure chord, and $A$ is finite, then $A$ is a proper near sequence.
\end{lemma}

\begin{proof}
 If $(A^*, B^*, C^*)$ is a pure chord relative to $H \le G$ then $\delta(A^*, B^*, C^*) = H$ and since $(A,B,C)$ is critical, we must have $|A^* \setminus A| < |H|$.  Since $A^*$ is finite, it is a proper $R$-sequence for some $R \in G/H$ and it follows immediately that $A$ is a near $R$-sequence.  Next suppose that $(A^*, B^*, C^*)$ is an impure chord relative to $H \le G$.  In this case, there exists a subgroup $K < H$ so that $K = \stab{A^*} = \stab{B^*} = \stab{C^*}$.  Now since $(A,B,C)$ is critical, it follows that $|A^* \setminus A| < |K|$.  Since $A^*$ is a proper fringed $R$-sequence for some $R \in G/H$, we again find that $A$ is a proper near $R$-sequence.
\end{proof}

\section{Proof}

In this section we prove Kemperman's Theorem.

\begin{proof}[Proof of Theorem~\ref{thm:kemperman}] \setcounter{claim}{0}
Suppose (for a contradiction) that the theorem fails and let $(A,B,C)$ be a counterexample with $|A| \le |B| \le |C|$ so that 

\begin{enumerate}
\item If there is a finite counterexample, then $|G|$ is minimum.
\item $\overline{C}$ is minimum (subject to 1).
\item the number of terms in $( [A], [B], [C] )$ equal to $G$ is maximum (subject to 1, 2).
\end{enumerate}
We shall establish properties of our trio with a series of claims.

\begin{claim}
 The group $H = \stab{A} = \stab{B} = \stab{C}$ is trivial. 
\end{claim}

Otherwise we obtain a smaller counterexample by passing to the quotient group $G/H$ and the trio $\big( \varphi_{G/H}(A), \varphi_{G/H}(B), \varphi_{G/H}(C) \big)$.  

\begin{claim}\label{clm:beat}
 None of the sets $A$, $B$, or $C$ is contained in a proper coset.  
\end{claim}
 
 Suppose for contradiction that $A,B$ or $C$ is contained in a proper coset, and apply Lemma~\ref{lem:beat}.  If $(A,B,C)$ is a pure beat, we have an immediate contradiction.  Otherwise, we may assume that $(A,B,C)$ is an impure beat relative to $H$ with continuation $(A',B',C')$.  If $H$ is finite, then $(A',B',C')$ contradicts our choice of $(A,B,C)$ for the first criteria.  Otherwise $H$ is infinite, and we may assume (without loss) that $C'$ is infinite.   
 
 Now, $A$ and $B$ are finite and $H$-quasiperiodic, so both $A$ and $B$ are both contained in a single $H$-coset.  Hence criteria (1) and (2) agree on $(A,B,C)$ and $(A',B',C')$.  Furthermore, only one term in $( [A], [B], [C] )$ is equal to $G$, but by construction, one of $A'$, $B'$ has closure $H$, and $[C'] = H$ (since $C'$ is cofinite).  Therefore $(A',B',C')$ is a counterexample which contradicts our choice.

\begin{claim}\label{clm:chord}
 $A$ is not a proper near sequence.
\end{claim}
 
Otherwise it follows from Lemma~\ref{lem:chord} that either $(A,B,C)$ is a pure or impure chord.  In the former case we have an immediate contradiction.  In the latter a continuation $(A',B',C')$ contradicts our choice of $(A,B,C)$.

\begin{claim}\label{clm:crithalf}
 If $D \subseteq G$ satisfies $(A,D)$ critical and $|D| > \frac{1}{2}|A|$ then $[D] = G$.
\end{claim}

Suppose (for a contradiction) that $(A,D)$ is critical and $|D| > \frac{1}{2}|A|$ and that $[D] = H + x$ for $H < G$.  If $A$ contains points in at least three $H$-cosets, then $A+D$ is $H$-quasiperiodic by Lemma~\ref{lem:discon_sum} so $|A+D| \ge 2|H| + |D| \ge 3|D| \ge |D| + |A|$ which contradicts the assumption that $(A,D)$ is critical.  It then follows from Claim~\ref{clm:beat} that $A$ must contain points in exactly two $H$-cosets.  Now, if $|A| \le |H|$ then we have $|A+D| \ge |H| + |D| \ge |A| + |D|$ which is contradictory.  Otherwise, $|A| > |H|$ and $A$ contains points in exactly two $H$-cosets, but then $A$ is a near $R$-sequence for some $R \in G/H$ and this contradicts Claim~\ref{clm:chord}.  

\begin{claim}\label{clm:nocoset_crit}
 There does not exist a nontrivial finite subgroup $H < G$ so that $(A,H)$ is critical.
\end{claim}

Suppose for contradiction that $(A,H)$ is critical with $\br{0} < H < G$. By Claim 1 we may choose an $H$-coset $R \in G/H$ so that $\emptyset \neq C \cap R \neq R$.  Now setting $C' = C \cup R$ and $B' = B \cap \third{A+R}$ our purification lemma implies $\delta(A,B',C') \ge \delta(A,B,C)$.  It follows from this that $0 \le  |C' \setminus C| - |B \setminus B'| < |H| - |B \setminus B'|$.  If $|A+H| = 2|H|$ then $A$ is a near sequence which contradicts Claim~\ref{clm:chord}.  Therefore, we have $|A| + |H| > |A+H| \ge 3|H|$.  This gives us $|B'| > |B| - |H| \ge |A| - |H| \ge \frac{1}{2}|A|$ so by Claim~\ref{clm:crithalf} we have that $[B'] = G$.

Now we let $(A^*,B^*,C^*)$ be a maximal supertrio of $(A,B',C')$.  Since $(A^*,B^*,C^*)$ is maximal with $\delta(A^*,B^*,C^*) \ge \delta(A,B',C') \ge \delta(A,B,C)$ and $|\overline{C^*}| < |\overline{C}|$, the theorem holds for $(A^*,B^*,C^*)$.  Therefore, $(A^*, B^*,C^*)$ must either be a pure or impure chord (since $[A] = [B] = [C] = G$).  Now Lemma~\ref{lem:near} 
implies that $A$ is a proper near sequence, but this contradicts Claim~\ref{clm:chord}.  

\begin{claim}\label{clm:deficiency1}
 Let $D \subseteq G$ be finite and assume that $(A,D)$ is nontrivial and critical.  Then $\delta(A,D) = 1$ and further, either $|D| = 1$ or $[D] = G$.
\end{claim}

 It follows immediately from Claim~\ref{clm:nocoset_crit} and Mann's theorem that $\delta(A,D) = 1$.  Suppose for contradiction that $[D] = H + x$ for $\br{0} < H < G$.  Then Lemma~\ref{lem:discon_sum} implies that $A+D$ is $H$-quasistable.  Since $[A] = G$, it follows that $H$ is finite.  Again, by Lemma~\ref{lem:discon_sum}, $\delta(A,H) \geq \delta(A,D)$, contradicting Claim~\ref{clm:nocoset_crit}.
 
\begin{claim}
 $B$ is not a Sidon set: $|(g+B) \cap B| > 1$ for some $g \in G \setminus \br{0}$.
\end{claim}

Suppose (for a contradiction) that $B$ is a Sidon set.  We must have $|A| \ge 3$ as otherwise either $[A] \neq G$ or $A$ is a near sequence.  Choose distinct elements $a_1,a_2,a_3 \in A$.  Now we have 
\[ |A+B| \ge | (a_1 + B) \cup (a_2 + B) \cup (a_3 + B) | \ge 3|B| - 3 \ge |A| + |B| \]
which is contradictory.

With this last claim in place, we are now ready to complete the proof.  Since $B$ is not a Sidon set, we may choose $g \in G \setminus \br{0}$ so that $B' = B \cap (g+B)$ satisfies $|B'| \ge 2$.  Set $C' = C \cup (-g + C)$ and $B'' = B \cup (g+B)$ and $C'' = C \cap (-g + C)$.  It now follows from basic principles that $(A,B',C')$ and $(A,B'',C'')$ are trios and
\begin{equation}\label{eqn:defsum}
 \delta(A,B',C') + \delta(A,B'',C'') = 2 \delta(A,B,C).
\end{equation}  
If $C'' = \emptyset$ then $G$ must be finite and we have $|C'| = 2|C|$, so
 \begin{eqnarray*}
  \delta(A,B',C') &=& |A| + |B'| + |C'| - |G|\\
                  &\ge& |A| + 2 + 2|C| - |G|  \\
                  &>& |A| + |B| + |C| - |G| \\
                  &=& \delta(A,B,C)
 \end{eqnarray*}
 which contradicts Claim~\ref{clm:deficiency1}.  Therefore $C'' \neq \emptyset$ and then both $(A,B',C')$ and $(A,B'',C'')$ are nontrivial.  Then, (\ref{eqn:defsum}) and Claim~\ref{clm:deficiency1} imply that $\delta(A,B',C') = \delta(A,B'',C'') = 1$ and that $(A,B',C')$ and $(A,B'',C'')$ are both maximal.  Since $|G \setminus C'| < |G \setminus C|$ the theorem holds for the trio $(A,B',C')$.  Since $|B'| \ge 2$, Claim~\ref{clm:deficiency1} implies that $[B'] = G$.  However, then $(A,B',C')$ must be a pure or impure chord, and then Lemma~\ref{lem:near} implies that $A$ is a near sequence which contradicts Claim~\ref{clm:chord}.  This completes the proof.
\end{proof}

\end{document}